\documentclass[12pt,reqno]{amsart}

\usepackage{amsthm, amsmath, amsfonts, amssymb, graphicx, hyperref}
\usepackage[all]{xy}

\newtheorem{theorem}{Theorem}[section]
\newtheorem{lemma}[theorem]{Lemma}

\theoremstyle{definition}
\newtheorem{definition}[theorem]{Definition}

\newtheorem{remark}[theorem]{Remark}

\newcommand\LSC{{\sf LSC}}
\newcommand\USC{{\sf USC}}

\newcommand\cl{{\sf cl}}

\newcounter{mycount}

\newcommand{\myref}[1]{\hyperref[#1]{#1}}

\begin{document}

\title{A new approach to the Kat\v{e}tov-Tong theorem}
\author{G.~Bezhanishvili}
\address{New Mexico State University}
\email{guram@nmsu.edu}

\author{P.~J.~Morandi}
\address{New Mexico State University}
\email{pmorandi@nmsu.edu}

\author{B.~Olberding}
\address{New Mexico State University}
\email{bruce@nmsu.edu}

\date{}


\begin{abstract}
We give a new proof of the Kat\v{e}tov-Tong theorem. Our strategy is to first prove the theorem for compact Hausdorff spaces, and then extend it to all normal spaces.
The key ingredient is how the ring of bounded continuous real-valued functions embeds in the ring of all bounded real-valued functions.
In the compact case this embedding can be described by an appropriate statement, which we prove implies both the Kat\v{e}tov-Tong theorem and
a version of the Stone-Weierstrass theorem.
We then extend the Kat\v{e}tov-Tong theorem to all normal spaces by showing how to extend upper and lower semicontinuous real-valued functions to the Stone-\v Cech compactification so that the less than or equal relation between the functions is preserved.
\end{abstract}

\maketitle

\section{Introduction}

For a topological space $X$ let $B(X)$ be the ring of all bounded real-valued functions and $C^*(X)$ the subring consisting of continuous functions. We recall (see, e.g., \cite[Def.~1, p.~360]{Bou89}) that $f\in B(X)$ is \emph{upper semicontinuous} if $f^{-1}(-\infty, \lambda)$ is open for all $\lambda \in \mathbb{R}$ and $f$ is \emph{lower semicontinuous} if $f^{-1}(\lambda, \infty)$ is open for all $\lambda \in \mathbb{R}$. It is well known (see, e.g., \cite[Prop.~3, p.~363]{Bou89}) that $f$ is upper semicontinous iff $f(x) = \inf_{U \in \mathcal{N}_x} \sup_{y \in U} f(y)$ and $f$ is lower semicontinuous iff $f(x) = \sup_{U \in \mathcal{N}_x} \inf_{y \in U} f(y)$, where $\mathcal{N}_x$ is the set of all open neighborhoods of $x$.
Let
\begin{align*}
\USC(X) &= \{ f \in B(X) \mid f \textrm{ is upper semicontinuous} \} \\
\LSC(X) &= \{ f \in B(X) \mid f \textrm{ is lower semicontinuous} \}.
\end{align*}

We can then formulate the famous Kat\v{e}tov-Tong theorem as follows:

\vspace{2mm}

\noindent {\bf Kat\v{e}tov-Tong Theorem (KT):} Let $X$ be a normal space. If $f \in \USC(X)$ and $g \in \LSC(X)$ with $f \le g$, then there is $h \in C^*(X)$ such that $f \le h \le g$.

\vspace{2mm}

Neither Kat\v{e}tov's proof \cite{Kat51,Kat53} nor Tong's \cite{Ton52} simplifies in the compact setting. We give a different proof
of (KT) by first proving it for compact Hausdorff spaces.
Our proof is based on \cite{BMO18c} where we gave a necessary and sufficient condition for a completely regular space $X$ to be compact in terms of how $C^*(X)$ embeds in $B(X)$. We also use Dilworth's characterization of upper and lower semicontinuous functions \cite[Lem.~4.1]{Dil50}.

To obtain the full version of (KT) for an arbitrary normal space $X$, we use the Stone-\v Cech compactification $\beta X$ of $X$. The key observation in this part of the proof is that if $f \in \USC(X)$, $g \in \LSC(X)$, and $f \le g$, then we can extend $f$ to $F \in \USC(\beta X)$ and $g$ to $G \in \LSC(\beta X)$ so that $F \le G$. This allows us to use the already established (KT) for compact Hausdorff spaces to produce a continuous function between $F$ and $G$, whose restriction to $X$ is then the desired continuous function on $X$.

We conclude the article by showing that a version of the Stone-Weierstrass theorem also follows from our approach.
In order to formulate the Stone-Weierstrass theorem, we point out that $B(X)$ is not only a ring, but an $\mathbb R$-algebra and $C^*(X)$ is an $\mathbb R$-subalgebra of $B(X)$. We recall that the \emph{uniform norm} is defined on $B(X)$ by $\|f\| = \sup f(X)$. We then have a metric space structure on $B(X)$, where the distance between $f,g$ is $\|f-g\|$. Elementary analysis arguments show that $B(X)$ and $C^*(X)$ are complete as metric spaces with respect to the uniform norm. If $X$ is compact Hausdorff, then $C^*(X)$ coincides with the $\mathbb R$-algebra $C(X)$ of all continuous real-valued functions on $X$.

\vspace{2mm}

\noindent {\bf Stone-Weierstrass Theorem (SW):} If $X$ is compact Hausdorff and $A$ is an $\mathbb{R}$-subalgebra of $C(X)$ which separates points of $C(X)$, then $A$ is uniformly dense in $C(X)$.

\vspace{2mm}

In addition to $B(X)$ being an $\mathbb R$-algebra, there is a natural order $\le$ on $B(X)$ defined by $f \le g$ iff $f(x) \le g(x)$ for all $x \in X$. It is elementary to see that the following conditions hold on $B(X)$:
\begin{enumerate}
\item $B(X)$ is a lattice;\footnote{That is, $f\vee g,f\wedge g$ exist in $B(X)$. They are defined by $(f\vee g)(x)=\max\{f(x),g(x)\}$ and
$(f\wedge g)(x)=\min\{f(x),g(x)\}$ for each $x\in X$.}
\item $f\le g$ implies $f+h \le g+h$;
\item $0 \leq f, g$ implies $0 \le fg$;
\item $0 \le f$ and $0\le \lambda\in\mathbb R$ imply $0\le \lambda f$.
\end{enumerate}
Thus, $B(X)$ is a lattice-ordered algebra or
$\ell$-algebra for short,
and $C^*(X)$ is an $\ell$-subalgebra of $B(X)$, where we recall that an $\ell$-subalgebra is an $\mathbb{R}$-subalgebra which is also a sublattice.
We can replace the $\mathbb R$-subalgebra condition in (SW) with an
$\ell$-subalgebra condition and arrive at the following version of the Stone-Weierstrass theorem.

\vspace{2mm}

\noindent {\bf Stone-Weierstrass for $\ell$-subalgebras (SW$_\ell$):} If $X$ is compact Hausdorff and $A$ is an $\ell$-subalgebra of $C(X)$ which separates points of $C(X)$, then $A$ is uniformly dense in $C(X)$.

\vspace{2mm}

We conclude the article by showing how to derive this version of the Stone-Weierstrass theorem from our approach.

\section{The Kat\v etov-Tong Theorem}

Let $X$ be a completely regular space. In \cite{BMO18c} we showed that $X$ is compact iff the inclusion $C^*(X) \subseteq B(X)$ satisfies Condition (C) below.
This will play an important role in proving (KT) for compact Hausdorff spaces.
To formulate (C), we point out that if $S\subseteq B(X)$ is bounded, then
the least upper bound $\bigvee S$ and the greatest lower bound $\bigwedge S$ exist in $B(X)$ and are pointwise.

\begin{definition}
Let $X$ be completely regular, $S,T\subseteq C^*(X)$ bounded, and $0<\varepsilon\in\mathbb R$.
\begin{enumerate}
\item[(C)] If $\bigwedge S + \varepsilon \le \bigvee T$ in $B(X)$, then there are finite $S_0 \subseteq S$ and $T_0 \subseteq T$ with
$\bigwedge S_0 \le \bigvee T_0$.
\end{enumerate}
\end{definition}

\begin{remark}
\begin{enumerate}
\item[]
\item The presence of $\varepsilon$ in (C) is necessary. To see this, identify $\mathbb{R}$ with a subalgebra of $B(X)$, and let $S = \{ \eta \in \mathbb{R} \mid 0 < \eta\}$ and $T = \{\lambda \in \mathbb{R} \mid \lambda < 0\}$. Then $\bigwedge S \le \bigvee T$ but there are no finite $S_0 \subseteq S$ and $T_0 \subseteq T$ satisfying $\bigwedge S_0 \le \bigvee T_0$.
\item In \cite{BMO18c} we considered Condition (C) for more general embeddings $A \to B(X)$. For the purposes of this paper we concentrate on the inclusion $C^*(X) \subseteq B(X)$.
\end{enumerate}
\end{remark}

\begin{theorem} \label{compact} $($\cite[Thm.~2.6(2)]{BMO18c}$)$
Let $X$ be completely regular. Then $X$ is compact iff the inclusion $C^*(X) \subseteq B(X)$ satisfies \emph{(C)}.
\end{theorem}

\begin{remark}
Our proof of (KT) for compact Hausdorff spaces (see Lemma~\ref{thm: C implies KT}) only needs one implication of Theorem~\ref{compact}, that the inclusion $C^*(X) \subseteq B(X)$ satisfies (C) if $X$ is compact.
\end{remark}

The following result uses Dilworth's lemma \cite[Lem.~4.1]{Dil50} characterizing upper and lower semicontinuous functions: Let $X$ be a completely regular space and $f \in B(X)$. Then $f \in \USC(X)$ iff $f$ is a pointwise meet of continous functions, and $f \in \LSC(X)$ iff $f$ is a pointwise join of continous functions.

\begin{lemma} \label{cor: C}
Let $X$ be compact Hausdorff, $f \in \USC(X)$, $g \in \LSC(X)$, and $\varepsilon > 0$. If $f + \varepsilon \le g$, then there is $a \in C(X)$
with $f \le a \le g$.
\end{lemma}

\begin{proof}
Since $X$ is compact Hausdorff, it follows from Urysohn's lemma that $X$ is completely regular. Therefore, by Dilworth's lemma, $f = \bigwedge S$ and $g = \bigvee T$ for some $S,T \subseteq C(X)$. Thus, by Theorem~\ref{compact}, there exist finite $S_0 \subseteq S$ and $T_0 \subseteq T$ with $f \le \bigwedge S_0 \le \bigvee T_0 \le g$. Set $a = \bigwedge S_0$. Then $a \in C(X)$ and $f \le a \le g$.
\end{proof}

We are ready to prove (KT) in the compact Hausdorff setting. For this we utilize a technique that goes back to Dieudonn\' e \cite{Die44}, and was used by Edwards \cite[p.~21]{Edw66} and Blatter and Seever \cite[pp.~32-33]{BS75}.

\begin{lemma} \label{thm: C implies KT}
Let $X$ be compact Hausdorff. If $f \in \USC(X)$ and $g \in \LSC(X)$ with $f \le g$, then there is $a \in C(X)$ such that $f \le a \le g$.
\end{lemma}

\begin{proof}
Let $f \in \USC(X)$,  $g \in \LSC(X)$, and $f \le g$. By induction we construct a sequence $\{a_n \mid n \ge 0\}$ in $C(X)$ such that for each $n \ge 1$,
\begin{align}
f - 1/2^n &\le a_n \le g \label{eqn4}\\
a_{n-1} - 1/2^{n-1} &\le a_{n} \le a_{n-1} + 1/2^{n-1}. \label{eqn5}
\end{align}
For the base case, since $(f-1/2) + 1/2 \le g$, by Lemma~\ref{cor: C} there is $a_1 \in C(X)$ with $f - 1/2 \le a_1 \le g$. Set $a_0 = a_1$. Then (\ref{eqn4}) and (\ref{eqn5}) are satisfied for $n = 1$. Suppose that $m \ge 1$ and we have
$a_0, \dots, a_m \in C(X)$ satisfying (\ref{eqn4}) and (\ref{eqn5}) for all $1 \le n \le m$. By (\ref{eqn4}) for $n = m$ we get $f \le a_m  + 1/2^m$. In addition, it is clear that $a_{m} - 1/2^{m+1} \le a_m  + 1/2^m$. Thus,
\[
f \vee (a_{m} - 1/2^{m+1}) \le a_m  + 1/2^m.
\]
Since $f, a_m\le g$, it is also clear that $f \vee (a_{m} - 1/2^{m+1}) \le g$. So
\[
f \vee (a_{m} - 1/2^{m+1}) \le g \wedge (a_m  + 1/2^m).
\]
Since $(a \vee b) + c = (a +c) \vee (b+c)$ holds in $B(X)$,
\begin{align*}
\left[(f - 1/2^{m+1}) \vee (a_{m} - 1/2^m)\right] + 1/2^{m+1} &=&  \\
(f - 1/2^{m+1} + 1/2^{m+1}) \vee (a_{m} - 1/2^m + 1/2^{m+1}) &=& \\
f \vee (a_{m} -1/2^{m+1}).
\end{align*}
Consequently,
\[
\left[(f - 1/2^{m+1}) \vee (a_{m} - 1/2^m)\right] + 1/2^{m+1} \le g \wedge (a_m + 1/2^m).
\]
By Lemma~\ref{cor: C}, there is $a_{m+1} \in C(X)$ satisfying
\[
(f - 1/2^{m+1}) \vee (a_{m} - 1/2^m) \le a_{m+1} \le g \wedge (a_m + 1/2^m).
\]
Therefore,
\begin{align*}
f - 1/2^{m+1} &\le a_{m+1} \le g \\
a_{m} - 1/2^m &\le a_{m+1} \le a_m + 1/2^m.
\end{align*}
Thus, (\ref{eqn4}) and (\ref{eqn5}) hold for $n = m+1$. By induction we have produced the desired sequence. Equation (\ref{eqn5}) implies that $\{a_n\}$ is a
Cauchy sequence, so has a uniform limit $a \in C(X)$. For each $x \in X$, (\ref{eqn4}) yields $f(x) - 1/2^n \le a_n(x) \le g(x)$ for each $n$. Taking limits as $n \to \infty$ gives $f(x) \le a(x) \le g(x)$. Therefore, $f \le a \le g$.
\end{proof}

To extend (KT) to an arbitrary normal space we require the following lemma.

\begin{lemma} \label{lem: lifting} $($\cite[Lem~7.2]{BMO18d}$)$
Let $X$ be a dense subspace of a compact Hausdorff space $Y$.
\begin{enumerate}
\item\label{lem: lifting(2)} If $f \in \USC(X)$, define $U(f)$ on $Y$ by
\[
U(f)(y) = \inf_{U \in \mathcal{N}_y} \sup_{x \in U\cap X} f(x).
\]
Then $U(f) \in \USC(Y)$ and extends $f$.

\item\label{lem: lifting(1)} If $f \in \LSC(X)$, define $L(f)$ on $Y$ by
\[
L(f)(y) = \sup_{U \in \mathcal{N}_y}\inf_{x \in U\cap X} f(x).
\]
Then $L(f) \in \LSC(Y)$ and extends $f$.
\end{enumerate}
\end{lemma}

\begin{remark}
There can exist upper semicontinuous extensions of $f \in \USC(X)$ to $Y$ other than $U(f)$. For example, let $X$ be an infinite discrete space and $Y$ an arbitrary compactification of $X$. If $f = 0$ on $X$, then $U(f) = 0$ on $Y$, while the characteristic function of $Y \setminus X$ is another upper semicontinuous extension of $f$ because $Y \setminus X$ is closed. Similarly, there can exist lower semicontinuous extensions of $f \in \LSC(X)$ to $Y$ other than $L(f)$.
\end{remark}

We are ready to prove (KT) for an arbitrary normal space $X$. Let $\beta X$ be the Stone-\v Cech compactification of $X$. Without loss of generality we may assume that $X$ is a subspace of $\beta X$. Since $X$ is normal, if $C, D$ are disjoint closed subsets of $X$, it is a simple consequence of Urysohn's lemma that $\cl_{\beta X}(C) \cap \cl_{\beta X}(D) = \varnothing$ (see, e.g., \cite[Cor.~3.6.4]{Eng89}).

\begin{theorem} \emph{(Kat\v{e}tov-Tong)} \label{thm: KT}
Let $X$ be a normal space. If $f \in \USC(X)$ and $g \in \LSC(X)$ with $f \le g$, then there is $h \in C^*(X)$ such that $f \le h \le g$.
\end{theorem}

\begin{proof}
Set $F = U(f)$ and $G = L(g)$. By Lemma~\ref{lem: lifting}, $F \in \USC(\beta X)$ and extends $f$, and $G \in \LSC(\beta X)$ and extends $g$.
We show $F \le G$. If not, there are $y \in \beta X$ and $\lambda, \eta \in\mathbb R$ with
$F(y) > \eta > \lambda > G(y)$. Let $U$ be an open neighborhood of $y$. Since $F(y) = \inf_{U \in \mathcal N_y} \sup_{x \in U\cap X} f(x)$, we have $\sup_{x \in U\cap X} f(x) > \eta$ for each open $U \in \mathcal N_y$. Therefore, $U \cap f^{-1}[\eta, \infty) \ne \varnothing$. Thus, $y \in \cl_{\beta X}f^{-1}[\eta, \infty)$. Similarly, $y \in \cl_{\beta X}g^{-1}(-\infty, \lambda]$. Because $f$ is upper semicontinuous, $f^{-1}[\eta, \infty)$ is closed and since $g$ is lower semicontinuous, $g^{-1}(-\infty, \lambda]$ is closed. As $X$ is normal,
$f^{-1}[\eta, \infty) \cap g^{-1}(-\infty, \lambda] \ne \varnothing$ since their closures in $\beta X$ are not disjoint. This is a contradiction to $f \le g$. Therefore, $F \le G$. By Lemma~\ref{thm: C implies KT}, there is $a \in C(\beta X)$ with $F \le a \le G$. If $h$ is the restriction of $a$, then $h \in C^*(X)$ and $f \le h \le g$.
\end{proof}

\begin{remark}
The key step in the proof of Theorem~\ref{thm: KT} is to show that if $f \in \USC(X)$ and $g \in \LSC(X)$ with $f \le g$, then $U(f) \le L(g)$.
\begin{enumerate}
\item If $X$ is not normal, it need not be true that $f \le g$ implies $U(f) \le L(g)$. To see this, let $X$ be a completely regular but not normal space.
Then there are disjoint closed sets $C, D$ of $X$ with $\cl_{\beta X}(C)$,  $\cl_{\beta X}(D)$ having nonempty intersection. Let $y \in \cl_{\beta X}(C) \cap \cl_{\beta X}(D)$, $f$ be the characteristic function of $C$, and $g$ the characteristic function of $X \setminus D$. It is easy to see that $f \in \USC(X)$ and $g \in \LSC(X)$.  Since $C$ and $D$ are disjoint, $f \le g$. Let $U$ be an open neighborhood of $y$. Then $U \cap C$ is nonempty, so $\sup_{x \in U \cap X} f(x) = 1$. Therefore, $U(f)(y) = 1$. Similarly, $U \cap D$ is nonempty, so $\inf_{x \in U \cap X} g(x) = 0$, and hence $L(g)(y) = 0$. Thus, $U(f) \not\le L(g)$.
\item We cannot replace $\beta X$ with an arbitrary compactification of $X$. For example, let $X$ be an infinite discrete space and $Y$ the one-point compactification of $X$. Let $A$ be an infinite subset of $X$ whose complement is also infinite,  $f$ the characteristic function of $A$, and $g = f$. Trivially $f \in \USC(X)$,  $g \in \LSC(X)$, and the same argument as above shows that $U(f)(\infty) = 1$ and $L(g)(\infty) = 0$. Thus, $U(f) \not\le L(g)$.
\end{enumerate}
\end{remark}

\section{The Stone-Weierstrass Theorem for $\ell$-subalgebras}

In this final section we show how to derive (SW$_\ell$) from (C).

\begin{definition} \label{def:closed and open}
Let $X$ be completely regular  and let $A$ be an $\ell$-subalgebra of $B(X)$.
\begin{enumerate}
\item Call $f \in B(X)$ \emph{closed} relative to $A$ if there is $S \subseteq A$ with $f = \bigwedge S$.
\item Call $g \in B(X)$ \emph{open} relative to $A$ if there is $T \subseteq A$ with $g = \bigvee T$.
\item Call $h \in B(X)$ \emph{clopen} relative to $A$ if $h$ is both closed and open relative to $A$.
\end{enumerate}
\end{definition}

\begin{remark}\label{translation}
\begin{enumerate}
\item[]
\item It is easy to see that if $S$ is a subset of $X$, then the characteristic function
$\chi_S$ is closed relative to $C^*(X)$ iff $S$ is a closed subset of $X$, and $\chi_S$ is open relative to
$C^*(X)$ iff $S$ is open in $X$. This motivates the terminology of Definition~\ref{def:closed and open}.
\item It is also easy to see that if $f \in B(X)$ is open (resp.~closed) relative to $A$,
then so is $f+\lambda$.
\end{enumerate}
\end{remark}

\begin{lemma} \label{clopen}
Let $X$ be compact Hausdorff. The clopen elements of $B(X)$ relative to $A$ are in the uniform closure of $A$ in $B(X)$.
\end{lemma}

\begin{proof}
Let $h$ be clopen in $B(X)$ relative to $A$ and let $\varepsilon  >0$. Then $(h + \varepsilon/2) + \varepsilon/2 \le h + \varepsilon$.
By Remark~\ref{translation}(2), $h + \varepsilon/2$ is closed and $h + \varepsilon$ is open relative to $A$. Therefore, $h + \varepsilon/2 = \bigwedge S$ and $h + \varepsilon = \bigvee T$ for some $S, T \subseteq A$. Since $X$ is compact Hausdorff, by Theorem~\ref{compact}, the inclusion $C(X) \subseteq B(X)$ satisfies Condition (C). Because $A \subseteq C(X)$, it follows from (C) that there are finite $S_0 \subseteq S$ and $T_0 \subseteq T$ with $h + \varepsilon/2 \le \bigwedge S_0 \le \bigvee T_0 \le h + \varepsilon$. Let $a = \bigwedge S_0$. Then $h + \varepsilon/2 \le a \le h + \varepsilon$, and $a \in A$ since $A$ is an $\ell$-subalgebra of $C(X)$. Thus, $\| a - h\| \le \varepsilon$. Since $\varepsilon$ is arbitrary, this shows
that $h$ is in the uniform closure of $A$.
\end{proof}

\begin{remark}
While we do not need it, the converse of Lemma~\ref{clopen} that each element in the uniform closure of $A$ is clopen relative to $A$ is also true (see, e.g., \cite[Lem~3.16(2)]{BMO19b}).
\end{remark}


The last ingredient needed for (SW$_\ell$) is the following lemma, the first two items of which are in \cite[Lem~2.8]{BMO18c}) and the third is an easy consequence of the first two.

\begin{lemma} \label{lem:closed and open}
Let $X$ be compact Hausdorff and let $A$ be an $\ell$-subalgebra of $C(X)$ which separates points of $X$.
\begin{enumerate}
\item If $f \in \USC(X)$, then $f$ is closed relative to $A$.
\item If $g \in \LSC(X)$, then $g$ is open relative to $A$.
\item If $h \in C(X)$, then $h$ is clopen relative to $A$.
\end{enumerate}
\end{lemma}

\begin{remark}
While we do not need it, the converse statements to the statements in Lemma~\ref{lem:closed and open} are true, and are easy to prove (see, e.g., \cite[Thm.~4, p.~362]{Bou89}).
\end{remark}

We are ready to prove (SW$_\ell$).

\begin{theorem}
If $X$ is compact Hausdorff and $A$ is an $\ell$-subalgebra of $C(X)$ which separates points of $C(X)$, then $A$ is uniformly dense in $C(X)$.
\end{theorem}

\begin{proof}
By Lemma~\ref{lem:closed and open},
elements of $C(X)$ are clopen relative to $A$, so lie in the uniform closure of $A$ in $B(X)$ by Lemma~\ref{clopen}. Thus, $A$ is dense in $C(X)$.
\end{proof}

\begin{remark}
While we have assumed that $A$ is an $\ell$-subalgebra of $B(X)$, it is sufficient to assume that $A$ is simply a vector sublattice of $B(X)$. Indeed, the existence of multiplication on $A$ is not needed in the proofs.
\end{remark}

\def\cprime{$'$}
\providecommand{\bysame}{\leavevmode\hbox to3em{\hrulefill}\thinspace}
\providecommand{\MR}{\relax\ifhmode\unskip\space\fi MR }
\providecommand{\MRhref}[2]{%
  \href{http://www.ams.org/mathscinet-getitem?mr=#1}{#2}
}
\providecommand{\href}[2]{#2}

\end{document}